\documentclass[english]{my_ccdconf}

\usepackage{amsmath,amssymb,amsthm}
\usepackage{url}

\newtheorem{theorem}{Theorem}
\newtheorem{proposition}[theorem]{Proposition}
\newtheorem{lemma}[theorem]{Lemma}

\newtheorem{example}[theorem]{Example}
\newtheorem{definition}[theorem]{Definition}
\newtheorem{remark}[theorem]{Remark}

\newcommand{\R}{\mathbb{R}}

\newcommand{\T}{\mathbb{T}}
\newcommand{\g}{\mathcal{L}}
\newcommand{\Epi}{\mbox{Epi}}

\newcommand{\Ga}{\gamma_1}
\newcommand{\Gb}{\gamma_2}

\begin{document}

\title{A unified approach to the calculus of variations on time scales}

\author{Ewa Girejko\aref{av1,bial},
Agnieszka B. Malinowska\aref{av2,bial},
Delfim F. M. Torres\aref{av3}}

\affiliation[av1]{Department of Mathematics, University of Aveiro,
    3810-193 Aveiro, Portugal
    \email{egirejko@ua.pt}}
\affiliation[av2]{Department of Mathematics, University of Aveiro,
    3810-193 Aveiro, Portugal
    \email{abmalinowska@ua.pt}}
\affiliation[av3]{Department of Mathematics, University of Aveiro,
    3810-193 Aveiro, Portugal
    \email{delfim@ua.pt}}
\affiliation[bial]{Faculty of Computer Science,
    Bia{\l}ystok University of Technology, 15-351 Bia\l ystok, Poland}

\maketitle


\begin{abstract}
In this work we propose a new and more general approach to the calculus of variations
on time scales that allows to obtain, as particular cases, both delta and nabla results.
More precisely, we pose the problem of minimizing or maximizing
the composition of delta and nabla integrals with Lagrangians
that involve directional derivatives.
Unified Euler-Lagrange necessary optimality conditions,
as well as sufficient conditions under appropriate
convexity assumptions, are proved. We illustrate
presented results with simple examples.\\

\noindent \textbf{About this Article:} This is a preprint of a paper
whose final form appeared in the CD-ROM Conference Proceedings of 2010 CCDC,
published by IEEE Industrial Electronics (IE) Chapter, Singapore:\\[0.1cm]
E. Girejko, A. B. Malinowska, D. F. M. Torres.
A unified approach to the calculus of variations on time scales,
Proceedings of 2010 CCDC, Xuzhou, China, May 26-28, 2010.
In: IEEE Catalog Number CFP1051D-CDR, 2010, 595--600.\\

\noindent \textbf{2010 Mathematics Subject Classification:} 49K05, 26E70, 34N05.
\end{abstract}


\keywords{Euler-Lagrange equations, calculus of variations,
delta and nabla calculi, time scales, directional derivatives.}

\footnotetext{This work was partially supported by the R\&D unit CIDMA,
via FCT and the EC fund FEDER/POCI 2010.
Girejko was also supported by the post-doc
fellowship SFRH/BPD/48439/2008; Malinowska by BUT,
via a project of the Polish Ministry of Science and
Higher Education ``Wsparcie miedzynarodowej mobilnosci naukowcow'';
Torres by the project UTAustin/MAT/0057/2008.}


\section{INTRODUCTION}

The theory of time scales was initiated by Aulbach and Hilger
in order to create a calculus that can unify and extend discrete
and continuous analysis \cite{A:H}.
It has found applications in several different fields that require simultaneous modeling
of discrete and continuous data, in particular in the calculus of variations.
There are two approaches that are followed in the literature
of the calculus of variations on time scales: one is concerned with the minimization of delta integrals
with a Lagrangian depending on delta derivatives \cite{A:M:02,B:T:08,B:04,F:T:08,AM:T};
the other with minimization of nabla integrals with integrands that involve nabla derivatives
\cite{A:B:L:06,A:U:08}. Both formulations of the problems of the calculus of variations
give results that are similar among them and similar to the classical results
of the calculus of variations (see, \textrm{e.g.}, \cite{Brunt})
but are obtained independently. The main goal of the present paper
is to give a unified treatment to the subject.
Motivated by this aim we propose the problem of the calculus of variations
on time scales that involves functionals with delta and nabla derivatives, \textrm{e.g.},
\begin{multline}
\label{problem:intro}
\text{extremize }\ \ \g(y)
= \Ga\int_a^b L_{\Delta}\left(t,y^\sigma(t),y^\Delta(t)\right) \Delta t\\
+ \Gb\int_a^b L_{\nabla}\left(t,y^\rho(t),y^\nabla(t)\right) \nabla t
\end{multline}
subject to the boundary conditions
$y(a) = \alpha$ and $y(b) = \beta$,
where $\alpha$ and $\beta$ are given real numbers. In the particular cases when $\Ga=0$ or $\Gb=0$
functional~\eqref{problem:intro} reduces to $\g(y) = \Gb\int_a^b L_{\nabla}\left(t,y^\rho(t),y^\nabla(t)\right) \nabla t$
or $\g(y) =\Ga\int_a^b L_{\Delta}\left(t,y^\sigma(t),y^\Delta(t)\right) \Delta t$.
More generally than this, we propose to unify delta and nabla calculus by using directional derivatives,
namely the derivative $D\overline{f}(t)(u)$ from the right of $\overline{f}$
at $t$ in the direction $u$, where $\overline{f}$ is the function defined
on a real interval and associated with $f$ (which is defined on a time scale), by the formula
\[
\overline{f}(t)=
\begin{cases}
f(t), &\text{if $t\in \T$}\, ,\\
f(s)+\frac{f(\sigma(s))-f(s)}{\mu(s)}(t-s), &\text{if $t\in(s,\sigma(s))$}\, ,
\end{cases}
\]
where $s\in \T$ is right-scattered.
With the use of the directional derivative we are able to prove unified Euler-Lagrange equations
and to give a unified treatment to the calculus of variations on time scales,
obtaining both delta and nabla results as trivial corollaries
and extending the calculus of variations
to a wider class of functions defined on time scales.
The paper is organized as follows.
Preliminary definitions and notations are gathered in Section~\ref{sec:prelim}.
The main results on the unification of problems of calculus of variations
are given in Section~\ref{E-L_equation}.
Finally, Section~\ref{sec:conc} presents some conclusions
and open questions.


\section{PRELIMINARIES}
\label{sec:prelim}

In this section we review necessary results from the literature. We assume
the reader to be familiar with the basic definitions and facts
concerning the delta and nabla differential calculus on time scales.
For an introduction to the subject we refer the reader to the books
\cite{B:P:01,B:P:03,Lak:book}.

Throughout the whole paper we assume $\T$ to be a given time scale with
$\inf\T:=a$, $\sup\T:=b$, and $I:=[a,b]\cap\T$ for $[a,b]\subset\R$.
Moreover, by $I^{\kappa}_{\kappa}$
(or $\T^{\kappa}_{\kappa}$) we mean
$I^{\kappa}_{\kappa}:=I^{\kappa}\cap I_{\kappa}$ (or, respectively,
$\T^{\kappa}_{\kappa}:=\T^{\kappa}\cap \T_{\kappa}$),
with $I^\kappa = I\setminus (\rho(b),b]$
and $I_\kappa = I\setminus [a,\sigma(a))$.
We recall that if $y$ is delta differentiable at $t\in\T$, then $y^\sigma(t) = y(t) +
\mu(t) y^\Delta(t)$; if $y$ is nabla differentiable at $t$, then
$y^\rho(t) = y(t) - \nu(t) y^\nabla(t)$.

If the functions $f,g : \mathbb{T}\rightarrow\mathbb{R}$ are delta
and nabla differentiable with continuous derivatives, then the
following formulas of integration by parts hold:
\begin{equation}
\label{intBP}
\begin{split}
\int_{a}^{b}f^\sigma(t) g^{\Delta}(t)\Delta t
&=\left.(fg)(t)\right|_{t=a}^{t=b}
-\int_{a}^{b}f^{\Delta}(t)g(t)\Delta t \, , \\
\int_{a}^{b}f(t)g^{\Delta}(t)\Delta t
&=\left.(fg)(t)\right|_{t=a}^{t=b}
-\int_{a}^{b}f^{\Delta}(t)g^\sigma(t)\Delta t \, , \\
\int_{a}^{b}f^\rho(t)g^{\nabla}(t)\nabla t
&=\left.(fg)(t)\right|_{t=a}^{t=b}
-\int_{a}^{b}f^{\nabla}(t)g(t)\nabla t \, ,\\
\int_{a}^{b}f(t)g^{\nabla}(t)\nabla t
&=\left.(fg)(t)\right|_{t=a}^{t=b}
-\int_{a}^{b}f^{\nabla}(t)g^\rho(t)\nabla t \, .
\end{split}
\end{equation}

The following fundamental lemma of the calculus of variations on
time scales involving a nabla derivative and a nabla integral has
been proved in \cite{NM:T}.

\begin{lemma}{\rm (The nabla Dubois-Reymond lemma \cite[Lemma~14]{NM:T}).}
\label{DBRL:n} Let $f \in C_{\textrm{ld}}(I, \mathbb{R})$. If
$$
\int_{a}^{b} f(t)\eta^{\nabla}(t)\nabla t=0
$$
for all $\eta \in C_{\textrm{ld}}^1(I, \mathbb{R})$
such that $\eta(a)=\eta(b)=0$, then $f(t) \equiv c$
for all $t\in I_\kappa$, where $c$ is a constant.
\end{lemma}

Lemma~\ref{DBRL:d} is the analogous delta version of
Lemma~\ref{DBRL:n}.

\begin{lemma}{\rm (The delta Dubois-Reymond lemma \cite[Lemma 4.1]{B:04}).}
\label{DBRL:d} Let $g\in C_{\textrm{rd}}(I, \mathbb{R})$. If
$$
\int_{a}^{b}g(t) \eta^\Delta(t)\Delta t=0
$$
for all $\eta\in C_{\textrm{rd}}^1(I, \mathbb{R})$ such that
$\eta(a)=\eta(b)=0$, then $g(t)\equiv c$ on $I^\kappa$
for some $c\in\mathbb \R$.
\end{lemma}

Proposition~\ref{prop:rel:der} gives a relationship between delta
and nabla derivatives.

\begin{proposition}{\rm (\cite[Theorems~2.5 and 2.6]{A:G:02}).}
\label{prop:rel:der} (i) If $f : \mathbb{T} \rightarrow \mathbb{R}$
is delta differentiable on $\mathbb{T}^\kappa$ and $f^\Delta$ is
continuous on $\mathbb{T}^\kappa$, then $f$ is nabla differentiable
on $\mathbb{T}_\kappa$ and
\begin{equation}
\label{eq:chgN_to_D} f^\nabla(t)=\left(f^\Delta\right)^\rho(t) \quad
\text{for all } t \in \mathbb{T}_\kappa \, .
\end{equation}
(ii) If $f : \mathbb{T} \rightarrow \mathbb{R}$ is nabla
differentiable on $\mathbb{T}_\kappa$ and $f^\nabla$ is continuous
on $\mathbb{T}_\kappa$, then $f$ is delta differentiable on
$\mathbb{T}^\kappa$ and
\begin{equation}
\label{eq:chgD_to_N} f^\Delta(t)=\left(f^\nabla\right)^\sigma(t)
\quad \text{for all } t \in \mathbb{T}^\kappa \, .
\end{equation}
\end{proposition}

\begin{proposition}{\rm (\cite[Theorem~2.8]{A:G:02}).}
\label{eq:prop} Let $a, b \in\mathbb{T}$ with $a \le b$ and let $f$
be a continuous function on $[a, b]$. Then,
\begin{equation*}
\begin{split}
\int_a^b f(t)\Delta t &= \int_a^{\rho(b)} f(t)\Delta t
+ (b - \rho(b))f^\rho(b) \, , \\
\int_a^b f(t)\Delta t &= (\sigma(a) - a) f(a)
+ \int_{\sigma(a)}^b f(t)\Delta t \, , \\
\int_a^b f(t)\nabla t &= \int_a^{\rho(b)} f(t)\nabla t
+ (b - \rho(b)) f(b) \, , \\
\int_a^b f(t)\nabla t &= (\sigma(a) - a) f^\sigma(a) +
\int_{\sigma(a)}^b f(t)\nabla t \, .
\end{split}
\end{equation*}
\end{proposition}

We end our brief review of the calculus on time scales with a
relationship between the delta and nabla integrals.

\begin{proposition}{\rm (\cite[Proposition~7]{G:G:S:05}).}
If function $f : \mathbb{T} \rightarrow \mathbb{R}$ is continuous,
then for all $a, b \in \mathbb{T}$ with $a < b$ we have
\begin{gather}
\int_a^b f(t) \Delta t = \int_a^b f^\rho(t) \nabla t \, , \label{eq:DtoN}\\
\int_a^b f(t) \nabla t = \int_a^b f^\sigma(t) \Delta t \, . \label{eq:NtoD}
\end{gather}
\end{proposition}


\section{MAIN RESULTS}
\label{E-L_equation}

Let $\mathbb{T}$ be a given time scale with $a, b \in \mathbb{T}$, $a < b$,
and $\mathbb{T} \cap (a,b) \ne \emptyset$;
$L_{\Delta}(\cdot,\cdot,\cdot)$ and $L_{\nabla}(\cdot,\cdot,\cdot)$ be two given smooth
functions from $\T \times \mathbb{R}^2$ to $\mathbb{R}$ and $\Ga,\Gb\in\R$.
The results of this section are trivially generalized for
admissible functions $y : \T\rightarrow\mathbb{R}^n$
but for simplicity of presentation
we restrict ourselves to the scalar case $n=1$.


\subsection{The delta-nabla calculus of variations}

We consider the delta-nabla integral functional
\begin{multline}
\label{eq:P}
\g(y)
= \Ga\int_a^b L_{\Delta}\left(t,y^\sigma(t),y^\Delta(t)\right) \Delta t\\
+ \Gb\int_a^b L_{\nabla}\left(t,y^\rho(t),y^\nabla(t)\right) \nabla t \, .
\end{multline}

One of our goals is to find the Euler-Lagrange equation
for $\g(y)$ defined by \eqref{eq:P}.
For simplicity of notation we introduce the operators
$[y]$ and $\{y\}$ defined by
\begin{equation*}
[y](t) = \left(t,y^\sigma(t),y^\Delta(t)\right) \, ,
\ \  \{y\}(t) = \left(t,y^\rho(t),y^\nabla(t)\right) \, .
\end{equation*}
Then,
\begin{equation*}
\begin{split}
\g_\Delta(y) &= \int_a^b L_\Delta[y](t) \Delta t \, , \\
\g_\nabla(y) &= \int_a^b L_\nabla\{y\}(t) \nabla t \, ,\\
\g(y) &= \Ga\g_\Delta(y)+\Gb \g_\nabla(y)\\
&=\Ga\int_a^b L_{\Delta}[y](t) \Delta t
+
\Gb\int_a^b L_\nabla\{y\}(t)  \nabla t  \, .
\end{split}
\end{equation*}

\begin{remark}
\label{obs} In the particular case $\Ga=0$
\eqref{eq:P} reduces to $\g(y) = \g_\nabla(y)$;
when $\Gb=0$ functional \eqref{eq:P} reduces to
$\g(y) = \g_\Delta(y)$.
\end{remark}

The delta-nabla problem of the calculus of variations
on time scales under our consideration consists of extremizing
\begin{equation}
\label{problem:P}
\mathcal{L}(y) = \Ga\int_a^b L_{\Delta}[y](t) \Delta t +
\Gb\int_a^b L_{\nabla}\{y\}(t) \nabla t
\end{equation}
in the class of functions $y \in C_{\diamond}^{1}(I, \mathbb{R})$,
where $C_{\diamond}^1$ denotes the class of functions
$y : I\rightarrow\mathbb{R}$  with
$y^\Delta$ continuous on $I^\kappa$
and $y^\nabla$ continuous on $I_\kappa$,
and satisfying the boundary conditions
\begin{equation}
\label{bou:con}
y(a) = \alpha \, , \quad y(b) = \beta \,
\end{equation}
with $\alpha$ and $\beta$ given real numbers.
A function $y \in C_{\diamond}^{1}(I, \mathbb{R})$ is said to be
\emph{admissible} provided it satisfies conditions \eqref{bou:con}.

\begin{definition}
\label{def:minimizer}
We say that $\hat{y}\in C_{\diamond}^{1}(I,
\mathbb{R})$ is a weak local minimizer (respectively weak local
maximizer) for problem \eqref{problem:P}--\eqref{bou:con} if there
exists $\delta >0$ such that
$\mathcal{L}(\hat{y})\leq \mathcal{L}(y)$ (respectively
$\mathcal{L}(\hat{y}) \geq \mathcal{L}(y)$)
for all $y \in C_{\diamond}^{1}(I, \mathbb{R})$ satisfying the boundary
conditions \eqref{bou:con} and
$\parallel y - \hat{y}\parallel_{1,\infty} < \delta$, where
$$
\parallel y\parallel_{1,\infty}:=
\parallel y^{\sigma}\parallel_{\infty}
+ \parallel y^{\rho}\parallel_{\infty} + \parallel
y^{\Delta}\parallel_{\infty} + \parallel
y^{\nabla}\parallel_{\infty}
$$
and
$\parallel y\parallel_{\infty} :=\sup_{t \in
I_{\kappa}^{\kappa}}\mid y(t) \mid$.
\end{definition}

Let $\partial_{i}L$ denote the standard
partial derivative of $L(\cdot,\cdot,\cdot)$
with respect to its $i$th variable, $i = 1,2,3$.
Theorem~\ref{thm:mr1} gives two different forms
for the Euler-Lagrange equation on time scales
associated with variational problem
\eqref{problem:P}--\eqref{bou:con}.

\begin{theorem}{\rm (The delta-nabla Euler-Lagrange equations
on time scales).}
\label{thm:mr1}
If $\hat{y} \in C_{\diamond}^{1}(I,\R)$ is a weak local extremizer of problem
\eqref{problem:P}--\eqref{bou:con}, then $\hat{y}$ satisfies
the following delta-nabla integral equations:
\begin{multline}
\label{eq:EL1}
\gamma_1
\left(\partial_3 L_\Delta[\hat{y}](\rho(t))
-\int_{a}^{\rho(t)} \partial_2 L_\Delta[\hat{y}](\tau) \Delta\tau\right)\\
+
\gamma_2\left(\partial_3 L_\nabla\{\hat{y}\}(t)
-\int_{a}^{t} \partial_2 L_\nabla\{\hat{y}\}(\tau) \nabla\tau\right)
= \text{const}
\end{multline}
for all $t \in I_\kappa$; and
\begin{multline}
\label{eq:EL2}
\gamma_1\left(\partial_3 L_\Delta[\hat{y}](t)
-\int_{a}^{t} \partial_2 L_\Delta[\hat{y}](\tau) \Delta\tau\right)\\
+
\gamma_2\left(\partial_3 L_\nabla\{\hat{y}\}(\sigma(t))
-\int_{a}^{\sigma(t)} \partial_2 L_\nabla\{\hat{y}\}(\tau) \nabla\tau\right) = \text{const}
\end{multline}
for all $t \in I^\kappa$.
\end{theorem}

\begin{proof}
Suppose that $\g$
has a weak local extremum at $\hat{y}$. We
consider the value of $\g$ at nearby
functions $\hat{y} + \varepsilon \eta$,
where $\varepsilon\in \mathbb{R}$ is a small parameter
and $\eta \in C_{\diamond}^{1}(I,\R)$ with $\eta(a)=\eta(b)=0$.
Thus, function $\phi(\varepsilon)
= \g(\hat{y} + \varepsilon \eta)$
has an extremum at $\varepsilon = 0$. Using the first-order
necessary optimality condition
$\left.\phi'(\varepsilon)\right|_{\varepsilon = 0} = 0$ we obtain:
\begin{multline}
\label{eq:prf:+}
\gamma_1\int_a^b
\left(\partial_2 L_\Delta[\hat{y}](t) \eta^\sigma(t)
+ \partial_3 L_\Delta[\hat{y}](t) \eta^\Delta(t)\right) \Delta t\\
+ \gamma_2\int_a^b
\left(\partial_2 L_\nabla\{\hat{y}\}(t) \eta^\rho(t)
+ \partial_3 L_\nabla\{\hat{y}\}(t) \eta^\nabla(t)\right)
\nabla t = 0 \, .
\end{multline}
Let
\begin{equation*}
A(t) = \int_a^t \partial_2 L_\Delta[\hat{y}](\tau) \Delta\tau \, , \ \
B(t) = \int_a^t \partial_2 L_\nabla\{\hat{y}\}(\tau) \nabla\tau \, .
\end{equation*}
Then, $A^\Delta(t) = \partial_2 L_\Delta[\hat{y}](t)$,
$B^\nabla(t) = \partial_2 L_\nabla\{\hat{y}\}(t)$,
and the first and third integration by parts formula in \eqref{intBP} tell us, respectively, that
\begin{equation*}
\begin{split}
\int_a^b \partial_2 & L_\Delta[\hat{y}](t) \eta^\sigma(t) \Delta t\\
&= \int_a^b A^\Delta(t) \eta^\sigma(t) \Delta t\\
&= \left. A(t) \eta(t)\right|_{t=a}^{t=b} - \int_a^b A(t) \eta^\Delta(t) \Delta t\\
&= - \int_a^b A(t) \eta^\Delta(t) \Delta t
\end{split}
\end{equation*}
and
\begin{equation*}
\begin{split}
\int_a^b \partial_2 & L_\nabla\{\hat{y}\}(t) \eta^\rho(t) \nabla t\\
&= \int_a^b B^\nabla(t) \eta^\rho(t) \nabla t\\
&= \left. B(t) \eta(t)\right|_{t=a}^{t=b}
- \int_a^b B(t) \eta^\nabla(t) \nabla t\\
&= - \int_a^b B(t) \eta^\nabla(t) \nabla t \, .
\end{split}
\end{equation*}
If we denote $f(t) = \partial_3 L_\Delta[\hat{y}](t) - A(t)$
and $g(t) = \partial_3 L_\nabla\{\hat{y}\}(t) - B(t)$,
then we can write the necessary optimality condition
\eqref{eq:prf:+} in the form
\begin{equation}
\label{eq:prf:+:aftIP}
\gamma_1\int_a^b f(t) \eta^\Delta(t) \Delta t
+ \gamma_2\int_a^b g(t) \eta^\nabla(t) \nabla t = 0 \, .
\end{equation}
We now split the proof in two parts:
(i) we prove \eqref{eq:EL1} transforming the delta integral
in \eqref{eq:prf:+:aftIP} to a nabla integral by means of
\eqref{eq:DtoN}; (ii) we prove \eqref{eq:EL2} transforming
the nabla integral in \eqref{eq:prf:+:aftIP} to a delta integral by means of \eqref{eq:NtoD}.

(i) By \eqref{eq:DtoN} the necessary optimality condition \eqref{eq:prf:+:aftIP} is equivalent to
\begin{equation*}
\int_a^b \left(\gamma_1f^\rho(t) (\eta^\Delta)^\rho(t)
+ \gamma_2 g(t) \eta^\nabla(t)\right) \nabla t = 0
\end{equation*}
and by \eqref{eq:chgN_to_D} to
\begin{equation}
\label{eq:bef:FL1}
\int_a^b \left(\gamma_1f^\rho(t)
+ \gamma_2g(t)\right)
\eta^\nabla(t) \nabla t = 0 \, .
\end{equation}
Applying Lemma~\ref{DBRL:n} to \eqref{eq:bef:FL1}
we prove \eqref{eq:EL1}:
\begin{equation*}
\label{eq:S12}
\gamma_1f^\rho(t)
+ \gamma_2g(t) = c \quad \forall t \in I_\kappa \, ,
\end{equation*}
where $c$ is a constant.

(ii) By \eqref{eq:NtoD} the necessary optimality condition \eqref{eq:prf:+:aftIP} is equivalent to
\begin{equation*}
\int_a^b \left(\gamma_1f(t) \eta^\Delta(t)
+ \gamma_2g^\sigma(t) \left(\eta^\nabla\right)^\sigma(t)\right) \Delta t = 0
\end{equation*}
and by \eqref{eq:chgD_to_N} to
\begin{equation}
\label{eq:bef:FL2}
\int_a^b \left(\gamma_1f(t)
+ \gamma_2g^\sigma(t)\right)
\eta^\Delta(t) \Delta t = 0 \, .
\end{equation}
Applying Lemma~\ref{DBRL:d} to \eqref{eq:bef:FL2}
we prove \eqref{eq:EL2}:
\begin{equation*}
\gamma_1f(t)
+  \gamma_2g^\sigma(t) = c \quad \forall t \in I^\kappa \, ,
\end{equation*}
where $c$ is a constant.
\end{proof}

\begin{example}
\label{example1}
Let $\T=\{1,3,4\}$, $\Ga,\Gb$ be arbitrary real numbers,
and consider the problem
\begin{equation}
\label{ex:3}
\begin{gathered}
\min\ \mathcal{L}(y) = \Ga\int^4_1t\left(y^\Delta(t)\right)^2\Delta t +
\Gb\int^4_1t\left(y^\nabla(t)\right)^2\nabla t
\\
y(1)=0,\ \ y(4)=1.
\end{gathered}
\end{equation}
Since
\[L_{\Delta}=t\left(y^{\Delta}\right)^2,\ \ \ L_{\nabla}=t\left(y^{\nabla}\right)^2,\]
we have
\[
\partial_2L_{\Delta}=0,\ \ \partial_3L_{\Delta}=2ty^{\Delta},
\ \ \partial_2L_{\nabla}=0,\ \ \partial_3L_{\nabla}=2ty^{\nabla}.
\]
Using equation~\eqref{eq:EL2} of Theorem~\ref{thm:mr1} we get
\begin{equation}\label{ex:1}
2\Ga ty^{\Delta}(t)+2\Gb \sigma(t)y^{\nabla}(\sigma(t))=C
\end{equation}
where $C\in\R$. By \eqref{eq:chgD_to_N} we can rewrite equation~\eqref{ex:1} in the form
\begin{equation}\label{ex:2}
2\Ga ty^{\Delta}(t)+2\Gb \sigma(t)y^{\Delta}(t)=C.
\end{equation}
Observe that $\Ga,\Gb$ cannot vanish simultaneously.
Solving equation~\eqref{ex:2} subject to the boundary conditions
$y(1)=0$ and $y(4)=1$ we get a candidate for a
local minimizer of problem \eqref{ex:3}:
\begin{equation}
\label{extremal:ex1}
y(t)=\begin{cases}
        0 & \text{if $t=1$}\\
        \frac{6\Ga+8\Gb}{7\Ga+11\Gb} & \text{if $t=3$}\\
        1 & \text{if $t=4$.}
        \end{cases}
\end{equation}
\end{example}

\begin{theorem}
\label{sc}
Let $ L_\Delta(\cdot,\cdot,\cdot)$ and $L_\nabla(\cdot,\cdot,\cdot)$
be jointly convex (concave) with respect to the second and third
argument for any $t\in I$, and $\Ga,\Gb\geq 0$. If $\hat{y} \in
C_{\diamond}^{1}(I,\R)$ is admissible and satisfies equation \eqref{eq:EL1}
(equivalently \eqref{eq:EL2}), then $\hat{y}$ is a global minimizer
(maximizer) of problem \eqref{problem:P}--\eqref{bou:con}.
\end{theorem}

\begin{proof}
We shall give the proof for the convex case. In this case we want to
show that the difference $\g(y) -\g(\hat{y})$ is greater or
equal than zero for any admissible $y$. Since $
L_\Delta(\cdot,\cdot,\cdot)$ and $L_\nabla(\cdot,\cdot,\cdot)$ are
jointly convex with respect to the second and third argument, we have
\begin{equation*}
\begin{split}
\g(y) -\g(\hat{y}) &= \Ga\int_a^b
\left(L_{\Delta}[y](t)-L_{\Delta}[\hat{y}](t)\right) \Delta t \\
&\quad + \Gb\int_a^b \left(L_\nabla\{y\}(t)-L_\nabla\{\hat{y}\}(t)\right)
\nabla t\\
&\geq \Ga\int_a^b \left[(y^\sigma(t)
-\hat{y}^\sigma(t))\partial_2L_{\Delta}[\hat{y}](t)\right.\\
&\quad +\left. (y^\Delta(t)
-\hat{y}^\Delta(t))\partial_3L_{\Delta}[\hat{y}](t)\right] \Delta
t\\
&\quad + \Gb\int_a^b \left[(y^\rho(t)
-\hat{y}^\rho(t))\partial_2L_\nabla\{\hat{y}\}(t)\right.\\
&\quad + \left. (y^\nabla(t)
-\hat{y}^\nabla(t))\partial_3L_\nabla\{\hat{y}\}(t)\right] \nabla t.
\end{split}
\end{equation*}
We can now proceed analogously to the proof of Theorem~\ref{thm:mr1}.
As result we get
\begin{multline*}
\g(y) -\g(\hat{y}) \\
\geq\int_a^b (y^\nabla(t)
-\hat{y}^\nabla(t))\Biggl[\Ga\Bigl(\partial_3
L_\Delta[\hat{y}](\rho(t))\\
\qquad - \int_{a}^{\rho(t)} \partial_2 L_\Delta[\hat{y}](\tau) \Delta\tau\Bigr)\\
+ \Gb\left(\partial_3 L_\nabla\{\hat{y}\}(t) -\int_{a}^{t}
\partial_2 L_\nabla\{\hat{y}\}(\tau) \nabla\tau\right)\Biggr]\nabla
t\\
+\left. A(t) (y(t)-\hat{y}(t))\right|_{t=a}^{t=b}+\left. B(t)
(y(t)-\hat{y}(t))\right|_{t=a}^{t=b}.
\end{multline*}
Clearly, the first term is equal to zero, since $\hat{y}$ is a
solution to the Euler-Lagrange equation \eqref{eq:EL1},
and the second and third terms
are also equal to zero since $y$ is admissible.
Therefore, $\g(y) \geq\g(\hat{y})$.
\end{proof}

\begin{example}
Consider again problem \eqref{ex:3} from Example~\ref{example1}
with $\T=\{1,3,4\}$. For fixed $\Ga,\Gb\geq0$ the assumptions
of Theorem~\ref{sc} are fulfilled and we conclude that \eqref{extremal:ex1}
is indeed the minimizer of \eqref{ex:3}.
\end{example}


\subsection{Calculus of Variations and Directional Derivatives}
\label{sec:direc deriv}

Let $\Box$ denote $\Delta$ or $\nabla$,
and $\xi$ denote $\sigma$ or $\rho$.
The proofs of Theorems~\ref{thm:mr1} and \ref{sc}
can be technically adapted to deal with the
more general variational problem
$$
\g(y) = \sum_{i=1}^{m} \int_a^b L_i\left(t,y^\xi(t),y^\Box(t)\right) \Box t
$$
or, even more general, to a functional given by the composition
of $m$ integrals:
\begin{multline*}
\g(y) = H\left(\int_a^b L_1\left(t,y^\xi(t),y^\Box(t)\right) \Box t,\right.\\
\left.\ldots, \int_a^b L_m\left(t,y^\xi(t),y^\Box(t)\right) \Box t\right),
\end{multline*}
where $H : \R^m\rightarrow \R$. We discuss here how to give
a precise unified treatment to each one of the terms
\begin{equation*}
\int_a^b L_i\left(t,y^\xi(t),y^\Box(t)\right)\Box t \, ,
\quad i = 1,\ldots,m \, .
\end{equation*}
For that we make use of directional derivatives. We begin by gathering
some basic definitions and notations.
Firstly, we recall the following general definition.
\begin{definition}
\label{def:epi}
Let $X$ be any nonempty subset of $\R$.
By \emph{the epigraph of} $f:X\rightarrow\R$,
denoted by $\Epi(f)$,  we mean the following set:
$\Epi(f):=\{(t,\lambda)\in X\times \R : f(t)\leq \lambda \}$.
\end{definition}

If $X=\T$ is a time scale, then we can rewrite the same definition
of epigraph and introduce the following extension of the epigraph
of a function  $f:\T\rightarrow\R$.
By $G(f)$ we denote the following set:
\begin{multline*}
G(f)=\bigcup_{t\in \T}\Bigl\{\alpha(t,y)+\beta(\sigma(t),z): \\
\, y\geq f(t), z\geq f^{\sigma}(t), \alpha+\beta=1, \alpha, \beta \geq 0\Bigr\}\, .
\end{multline*}
Let $X=I$.
Using the formulation of $G(f)$  we can assign to $f:X\rightarrow\R$
a new function $\overline{f}:[a,b]\rightarrow \R$ by the condition
\begin{equation}
\label{eq:f bar}
\Epi(\overline{f})=G(f).
\end{equation}
Let us notice that for $f,g:X\rightarrow \R$ and $a,b\in\R$
the following holds: $a\overline{f}+b\overline{g}=\overline{af+bg}$.

\begin{remark}
\label{rem:epi}
Function $\overline{f}$ defined by formula \eqref{eq:f bar}
can be presented in the following way (see, \textrm{e.g.}, \cite{Di}):
\[\overline{f}(t)=
\begin{cases}
f(t), &\text{if $t\in \T$} \, ,\\
f(s)+\frac{f(\sigma(s))-f(s)}{\mu(s)}(t-s), &\text{if $t\in(s,\sigma(s))$}\, ,
\end{cases}\]
when $s\in \T$ is right-scattered; or
\[\overline{f}(t)=
\begin{cases}
f(t), &\text{if $t\in \T$}\, ,\\
f(s)+\frac{f(s)-f(\rho(s))}{\nu(s)}(t-s), &\text{if $t\in(\rho(s),s)$}\, ,
\end{cases}
\]
when $s\in \T$ is left-scattered.
\end{remark}

\begin{proposition}[\cite{G:M:W:09}]
Let $f:I\rightarrow\R$. Then the following statements are equivalent:

a) The set $G(f)$ is convex;

b) $\overline{f}$ is convex in $[a,b]$;

c)  $f$ is convex in $X$.

\end{proposition}

\begin{definition}\label{def:direct}
Let $[a,b]$ be a real interval and let $\overline{f}:
[a,b]\to \mathbb{R}$ be defined by formula~(\ref{eq:f bar}).
We say that the function defined by
\begin{equation}\label{eq:deriv}
 \ \  D\overline{f}(t)(u)\mathrel{\mathop:}
 =\lim_{h\to 0^+} \displaystyle {\overline{f}(t+hu)-\overline{f}(t)\over h}
\end{equation}
is the \emph{derivative from the right of $\overline{f}$
at $t$ in the direction $u$} if the limit
of the right-hand side of (\ref{eq:deriv}) exists.
If $D\overline{f}(t)(u)$ exists for all directions $u$, we say that
$\overline{f}$ \emph{is differentiable from the right at} $t$.
\end{definition}

Let us recall the following useful relations between delta and nabla derivatives
of $f$ at point $t$ (if they exist) and the derivative
from the right of the corresponding function $\overline{f}$.

\begin{proposition}[\cite{G:M:W:09}]
\label{prop:Dupf+deriv}
Let $t\in\T^{\kappa}_{\kappa}$ and $f:\T\rightarrow\R\cup\{\pm\infty\}$.
(a) If  $f^{\Delta}(t)$ exists, then
\[D\overline{f}(t)(u)=uf^{\Delta}(t)\ \ \text{for }\ u\geq 0\,.\]

b)  If  $f^{\nabla}(t)$ exists, then
\[D\overline{f}(t)(u)=uf^{\nabla}(t)\ \ \text{for }\ u\leq0\,.\]
\end{proposition}

\begin{remark}
\label{eq:6}
When we fix $u=1$ ($u=-1$) then immediately
from Proposition \ref{prop:Dupf+deriv} one gets
\begin{enumerate}
  \item  $D\overline{f}(t)(1)=f^{\Delta}(t)=\overline{f}'_+(t)$,
  \item  $D\overline{f}(t)(-1)=-f^{\nabla}(t)=-\overline{f}'_-(t)$,
\end{enumerate}
where  $\overline{f}'_+(t)$ and $\overline{f}'_-(t)$ denote
left and right hand side derivatives of $\overline{f}$ in the classical sense.
\end{remark}

We introduce the following notations and definitions.

\begin{definition}\label{def:uni}
Let $u\in\R$ be any real number. Then,
\begin{equation*}
d_ut:=\\
\left\{
  \begin{array}{ll}
    u\Delta t, \text{ if $u\geq 0$} \\
    u\nabla t, \text{ if $u\leq 0$}
  \end{array}
\right.
\end{equation*}
\begin{equation*}
 y\circ \xi_u:=\\
  \left\{
  \begin{array}{ll}
    u(y\circ \sigma), \text{ if $u\geq 0$} \\
    u(y\circ\rho), \text{ if $u\leq 0$}.
  \end{array}
\right.
\end{equation*}
\end{definition}

\begin{remark}
With the notation of Definition~\ref{def:uni} we have
\begin{equation*}
\int_a^b \overline{f}(t)d_ut:=\\
\left\{
  \begin{array}{ll}
    u\int_a^b f(t)\Delta t, \text{ if $u\geq 0$} \\
    u\int_a^b f(t)\nabla t, \text{ if $u\leq 0$}
  \end{array}
\right.
\end{equation*}
where $\overline{f}$ is defined by formula~\eqref{eq:f bar}.
\end{remark}

Let us consider the following problem.
Given $u\in \R \setminus \{0\}$, find $y$ that is a solution to
\begin{equation}
\label{problem:Pu}
\begin{gathered}
\min\ \mathcal{L}(y) = \int_a^b L(t, (y\circ \xi_u)(t),D\overline{y}(t)(u))d_ut \, ,\\
y(a)=\alpha,\quad y(b)=\beta \, ,
\end{gathered}
\end{equation}
in the class of functions $y \in C_{\diamond}^{1}(I,\R)$.

\begin{remark}
We are excluding the case $u = 0$ for which problem
\eqref{problem:Pu} is trivial (for $u = 0$ there is
nothing to minimize).
\end{remark}

\begin{remark}
Proposition~\ref{prop:Dupf+deriv} implies the following:
if $y$ is $\Delta$-differentiable, then for $u=1$ \eqref{problem:Pu}
is just a problem of the calculus of variations
with $\Delta$ derivative (see \cite{B:T:08,B:04}),
while if $f$ is $\nabla$-differentiable,
then for $u=-1$ \eqref{problem:Pu} reduces to a problem
of the calculus of variations with
$\nabla$ derivative (see \cite{A:B:L:06,NM:T}).
\end{remark}

\begin{definition}
\label{def:minimizer:uni}
We say that $\hat{y}\in C_{\diamond}^{1}(I,\mathbb{R})$
is \emph{a weak local minimizer} (respectively \emph{weak local
maximizer}) for problem \eqref{problem:Pu} if there
exists $\delta
>0$ such that
$$
\mathcal{L}(\hat{y})\leq \mathcal{L}(y) \quad (\text{respectively} \
\   \mathcal{L}(\hat{y}) \geq \mathcal{L}(y))
$$
for all $y \in C_{\diamond}^{1}(I, \mathbb{R})$ satisfying
$||y-\hat{y}||_{1,\infty}<\delta$.
\end{definition}

\begin{theorem}{\rm (The directional Euler-Lagrange equation
on time scales).}
\label{thm:uni EL}
If $y \in C_{\diamond}^{1}(I,\R)$ is a weak local minimizer to problem
\eqref{problem:Pu}, then $y$ satisfies
the following equation:
\begin{multline}
\label{uni EL}
D\left(\partial_3 L(t, (y\circ \xi_u)(t),D\overline{y}(t)(u))\right)(u)\\
=u\cdot\partial_2 L(t, (y\circ \xi_u)(t),D\overline{y}(t)(u)),
\ \forall t\in I^{\kappa^2}_{\kappa^2}\, ,
\end{multline}
where $\overline{y}$ is defined by formula~\eqref{eq:f bar}.
\end{theorem}

\begin{proof}
We consider two cases: $u>0$ and $u<0$. For $u>0$ problem~\eqref{problem:Pu}
reduces to
\begin{equation}
\label{eq:uni:delta}
\begin{gathered}
\min\ \int_a^b uL(t, u(y\circ \sigma)(t),uy^{\Delta}(t))\Delta t,\\
y(a)=\alpha,\quad y(b)=\beta\, .
\end{gathered}
\end{equation}
If we set $f(t,y^{\sigma}(t),y^{\Delta}(t)):=uL(t, u(y\circ \sigma)(t),uy^{\Delta}(t))$
then problem~\eqref{eq:uni:delta} is equivalent to
\begin{equation}
\label{eq:int}
\min\ \int_a^b f(t,y^{\sigma}(t),y^{\Delta}(t))\Delta t\, , \ \ y(a)=\alpha,\ y(b)=\beta.
\end{equation}
For problem~\eqref{eq:int} the Euler-Lagrange equation \eqref{eq:EL2}
with $\Ga = 1$ and $\Gb = 0$ gives the delta equation
\begin{equation*}
\left[\partial_3f(t,y^{\sigma}(t),y^{\Delta}(t))\right]^{\Delta}
=\partial_2f(t,y^{\sigma}(t),y^{\Delta}(t)),
\end{equation*}
which is equivalent to
\begin{equation*}
\left[\partial_3L(t, u(y\circ \sigma)(t),uy^{\Delta}(t))\right]^{\Delta}
=\partial_2L(t, u(y\circ \sigma)(t),uy^{\Delta}(t))
\end{equation*}
for every $t\in I^{\kappa^2}$, \textrm{i.e.}, we obtain \eqref{uni EL} for $u>0$.\\
Similarly, let us take $u<0$. Then problem~\eqref{problem:Pu}
reduces to the following nabla problem of the calculus of variations:
\begin{equation}\label{eq:uni:nabla}
\min\ \int_a^b uL(t, u(y\circ \rho)(t),uy^{\nabla}(t))\nabla t,\ \ \ y(a)=\alpha,\ y(b)=\beta.
\end{equation}
If we set $g(t,y^{\sigma}(t),y^{\nabla}(t)):=uL(t, u(y\circ \rho)(t),uy^{\nabla}(t))$
then problem~\eqref{eq:uni:nabla} is equivalent to
\begin{equation*}
\min\ \int_a^b g(t,y^{\rho}(t),y^{\nabla}(t))\nabla t,\ \ \ y(a)=\alpha,\ y(b)=\beta.
\end{equation*}
From the Euler-Lagrange equation \eqref{eq:EL1}
with $\Ga = 0$ and $\Gb = 1$ we get the nabla differential equation
\begin{equation*}
\left[\partial_3g(t,y^{\rho}(t),y^{\nabla}(t))\right]^{\nabla}
=\partial_2g(t,y^{\nabla}(t),y^{\nabla}(t))
\end{equation*}
that one can write equivalently as
\begin{equation*}
\left[\partial_3L(t, u(y\circ \rho)(t),uy^{\nabla}(t))\right]^{\nabla}
=\partial_2L(t, u(y\circ \rho)(t),uy^{\nabla}(t))
\end{equation*}
for every $t\in I_{\kappa^2}$, \textrm{i.e.}, we obtain \eqref{uni EL} for $u<0$.
\end{proof}


\section{CONCLUSION}
\label{sec:conc}

We introduce general problems of the calculus of variations
on time scales that unify the delta and the nabla problems
previously studied in the literature.
The proposed calculus of variations
extends the problems with delta derivatives
considered in \cite{B:04} and analogous nabla
problems \cite{A:B:L:06} to more general cases
described by the composition of delta and/or nabla
integrals or, even more generally, to the composition
of variational integrals with directional derivatives:
\begin{multline*}
\g(y) = H\left(\int_a^b L_1(t, (y\circ \xi_{u_1})(t),D\overline{y}(t)(u_1))d_{u_1}t,\right.\\
\left.\ldots, \int_a^b L_m(t, (y\circ \xi_{u_m})(t),D\overline{y}(t)(u_m))d_{u_m}t\right),
\end{multline*}
where $H : \R^m\rightarrow \R$ and $u = \left(u_1,\ldots,u_m\right) \in \R^m$.
We prove Euler-Lagrange type conditions
for the generalized calculus of variations
as well as sufficient conditions under proper
convexity assumptions. We claim that the notion of directional
derivative plays an important role
in the calculus of variations on time scales.
More than that, we hope the notion of directional derivative
will become a standard tool in the theory of time scales.
It would be interesting to generalize our results
to variational problems involving higher-order
directional derivatives, unifying and extending
the higher-order results on time scales
of \cite{F:T:08} and \cite{NM:T}.
This is a question needing further developments.



\end{document}